\documentclass{amsart}

\usepackage{amsfonts, amsmath,amscd, amssymb, cite, color, latexsym, mathrsfs, mathtools,
slashed, stmaryrd, verbatim, wasysym, amsthm }
\usepackage[all]{xy}
\usepackage[pdftex]{graphicx}
\usepackage{hyperref}
\usepackage[applemac]{inputenc}
\usepackage{tikz-cd}
\usepackage{oubraces}
\usepackage{cancel}
\usepackage{float}
\allowdisplaybreaks
\usepackage{tensor}

\usepackage{subfig}

\usepackage{cite}

\newtheorem{thm}{Theorem}[section]

\newtheorem{prop}[thm]{Proposition}
\newtheorem{cor}[thm]{Corollary}

\newtheorem{definition}[thm]{{Definition}}

\DeclareMathOperator*{\esssup}{ess\,sup}
\DeclareMathOperator*{\essinf}{ess\,inf}

\renewcommand{\bar}{\overline}

\newcommand{\wt}[1]{\widetilde{#1}}

\newcommand\pa{\partial}

\newcommand\eps\varepsilon
\renewcommand\epsilon\varepsilon

\newcommand\CI{{\mathcal{C}}^{\infty}}

\newcommand\Div{\operatorname{div}}

\newcommand\loc{\operatorname{loc}}

\newcommand\dist{\operatorname{dist}}

\newcommand\paperintro%
        {%
         }
\newcommand\paperbody%
        {%
         }

\newcommand\bbC{\mathbb{C}}
\newcommand\bbD{\mathbb{D}}
\newcommand\bbE{\mathbb{E}}

\newcommand\bbP{\mathbb{P}}

\newcommand\bbR{\mathbb{R}}

\newcommand\cA{\mathcal{A}}

\newcommand\cC{\mathcal{C}}

\newcommand\cK{\mathcal{K}}
\newcommand\cL{\mathcal{L}}

\DeclareMathAlphabet{\mathpzc}{OT1}{pzc}{m}{it}

\title[2d Exponential localization via Brownian Flux]{Exponential localization of 2d Magnetic Schr\"{o}dinger eigenfunctions via Brownian Flux}
\author{Hadrian Quan}

\begin{document}

\maketitle

\vspace*{-1cm}\begin{abstract}
    We study solutions of $\tfrac12 (-i\nabla - A(x))^2f=\lambda f$ on domains $\Omega\subset \bbR^2$ with Dirichlet boundary conditions and prove exponential decay estimates in terms of an Agmon type distance to a classically allowed region. This metric depends only on the eigenvalue and associated magnetic field. In fact the main quantity in the weight for this distance function can be interpreted as an `average magnetic flux' of the magnetic field along all domains whose boundary is the closed curved formed by joing the `minimal path' to another random path. Our estimates are based on an analysis of the associated heat kernel using the Feynman-Kac-It\'o formula. 
\end{abstract}

\section{Introduction and results}

\subsection{Agmon estimates for scalar potentials} This paper is a contribution to the vast literature on the exponential decay of eigenfunctions of purely magnetic Schr\"odinger operators, $H(A)=\tfrac12(-i\nabla-A(x))^2$, associated to the discrete spectrum \cite{arnold2019localization, brummelhuis1991exponential, bonnaillie2022purely, helffer-nourrigat, fournais-helffer, lieb1976bounds, montgomery, poggi, poggi-mayboroda}. There is an even more expansive history of works which study the case of a Schr\"odinger operator with a scalar (or electric) potential, $-\tfrac12\Delta+V$,  \cite{agmon, carmona1978pointwise, carmona-simon, filoche2012universal, simon1974pointwise, simon1983semiclassical, lithner1964theorem}, but the mechanism for localization in the magnetic case is comparatively less well-understood. Much of the existing work in the magnetic cases proceeds along similar lines as Agmon's pioneering approach to the scalar potential case, which we briefly recall.

Starting from the Schr\"odinger eigenvalue problem for a scalar potential $V$,
\[ -\tfrac12\Delta f + Vf = \lambda f, \]
for a $V\geq 0$ which grows sufficiently fast at infinity, one multiplies by $f$ and integrates by parts to obtain
\[ \tfrac12\int_{\bbR^d} |\nabla f|^2 dx + \int_\Omega V\,f^2 dx = \int \lambda f^2 dx . \]
This identity suggests that most of the $L^2$-mass of an eigenfunction lies in the `classically allowed region' $E_\lambda=\{x: V(x)\leq \lambda\}$ and relatively little $L^2$-mass of an eigenfunction can be contained in the `classically forbidden region' $\{x\in \bbR^d: V(x)>\lambda\}$. To refine this heuristic one can observe from this energy identity immediately that
\[ \tfrac12\int_{\bbR^d} |\nabla (e^\phi f)|^2 dx + \int_\Omega e^{2\phi} (V-|\nabla \phi|^2-\lambda) f^2 dx = 0,  \] 
and thus for \emph{any} $\phi$ one concludes after dropping the non-negative term
\[ \int_\Omega e^{2\phi} (V-|\nabla \phi|^2-\lambda) f^2 dx \leq 0 , \]
at which point one seeks a clever choice of $\phi$ to imply localization. Agmon's now classical idea was to consider a weighted distance to the classically allowed region, by defining
\begin{equation}\label{eq:scalar-agmon}
    \rho_{V,\lambda}(x,y) = \inf_{\gamma} \int_0^1 \max\left(\sqrt{V(\gamma(t))-\lambda},0\right) |\dot\gamma(t)| dt  
\end{equation} 
and taking $\phi=(1-\eps)\rho_{V,\lambda}(x,E_\lambda)$. This immediately implies that the eigenfunction is localized inside a small neighborhood of the classical region and the decay rate outside of the classical region can be measured by this weighted distance to the classical region.

Given the effectiveness and simplicity of this approach it is no surprise many authors have also tried to construct analogues of a magnetic Agmon metric via similar energy identities. In practice this is quite difficult as the naive approach of beginning with the eigenvalue equation $H(A)f=\lambda f$ and then integrating by parts leads to the identity
\[ \tfrac12\int_{\bbR^d} |(-i\nabla-A) f|^2 dx = \int \lambda f^2 dx . \]%
Now that the magnetic potential is part of the `kinetic energy' one cannot drop this non-negative term; any localization induced by the magnetic field must be implicit in this left hand quantity. This naive approach thus leads to decay estimates which are the same as those of $A\equiv 0$, which are known to be non-optimal; this can already be seen in the case of a constant magnetic field (see \S \ref{sec:phys-levy}). The aforementioned works remedy this issue by a combination of assumptions on the structure of the magnetic potential/field to be e.g. non-negative, radial, polynomial of some order, possessing a unique non-degenerate minimum etc., and combining such assumptions with a variety of methods ranging from semiclassical to harmonic analysis. 

\subsection{Probability and spectral theory} On the other hand, there is similarly a long history employing probabilistic methods to study Schr\"odinger operators, going all the way back to the foundational work of Kac \cite{kac1951some}, and reinvigorated by the book of Simon \cite{simon1979functional}, which built off and inspired \cite{carmona1978pointwise, carmona-simon, lieb1976bounds, malliavin1986minoration, simon1983semiclassical, shigekawa1987eigenvalue, steinerberger2017localization, steinerberger2021effective, ueki} among many others. This probabilistic approach is gratifying given the ubiquity of the `path integral' interpretation of quantum mechanics, which motivates heuristics for the structure of Schr\"odinger eigenfunctions; Agmon's metric itself is a prime example of the behavior of eigenfunctions being dictated by the structure of ``classical paths". 

The present paper proceeds in this vein by introducing a magnetic Agmon metric obtained by probabilistic methods. Rather than the potential $A\in \CI(\Omega,\bbR^2)$, the function playing a role in the analogous Agmon weight is the magnetic field 
\[ \beta = \text{curl}\, A :=\pa_{x_1} A_2-\pa_{x_2} A_1\in \CI(\Omega). \] 
This is unsurprising given the natural gauge invariance our operator enjoys, see e.g. \S \ref{sec:gauge}. Further, the `magnetic flux' associated to $\beta$ is of primary physical interest, and appears in e.g. the celebrated Aharonov-Bohm effect.

The form of our magnetic Agmon metric will resemble \eqref{eq:scalar-agmon}, with the role of $V(\gamma(t))$ replaced by a quantity which can loosely be interpreted as the `average magnetic flux' of $\beta$ integrated over domains bounded by fixing a $\gamma(t)$ and closing it with `every other path' starting at $\gamma(0)$. Let us now introduce our main quantities and results. 

\subsection{Statement of the result}

Let $\omega_x(t)$ denote a Brownian motion on $\bbR^2$ started at $x\in \Omega$, i.e. $\omega_x(0)=x$. For $\gamma: [0,T]\to \Omega$, an arbitrary $C^1$ path from $\gamma(0)=x$ to $\gamma(1)=y$, we can define a straight-line homotopy,
\[ \Gamma:[0,1]_s\times [0,T]_t \to \Omega, \quad \Gamma(s,t) = s \omega_x(t) + (1-s)\gamma(t)  \]
which continuously connects $\omega_x(t)$ to $\gamma(t)$ as the deformation parameter $s$ varies. Having fixed this notation, we now define for each $C^1$ path $\gamma:[0,T]\to \Omega$ the function,
\begin{equation}\label{eq:avg-curl}
    \bar\beta(\gamma(t)) = \bbE\left[ |\omega_x(t)-\gamma(t)|^2  \left(\int_0^1 \beta(\Gamma(s,t)) \,s\, ds\right)^2 \right]  .
\end{equation}
For each fixed time $t$ this quantity should be thought of as a `weighted distance' between $\gamma$ and the location of a `random path' at the time; this weighting is by the integral of the magnetic field along the straight-line connecting them. Integrating the function $\bar\beta(\gamma(t))$ along $[0,T]$, and using Tonelli to exchange integration and expectation we obtain a quantity which is (morally) the square of the expected flux of $\beta$ between $\gamma(t)$ and any $\omega_x(t)$.

\begin{figure}[h]
\centering
    \includegraphics[scale=.28]{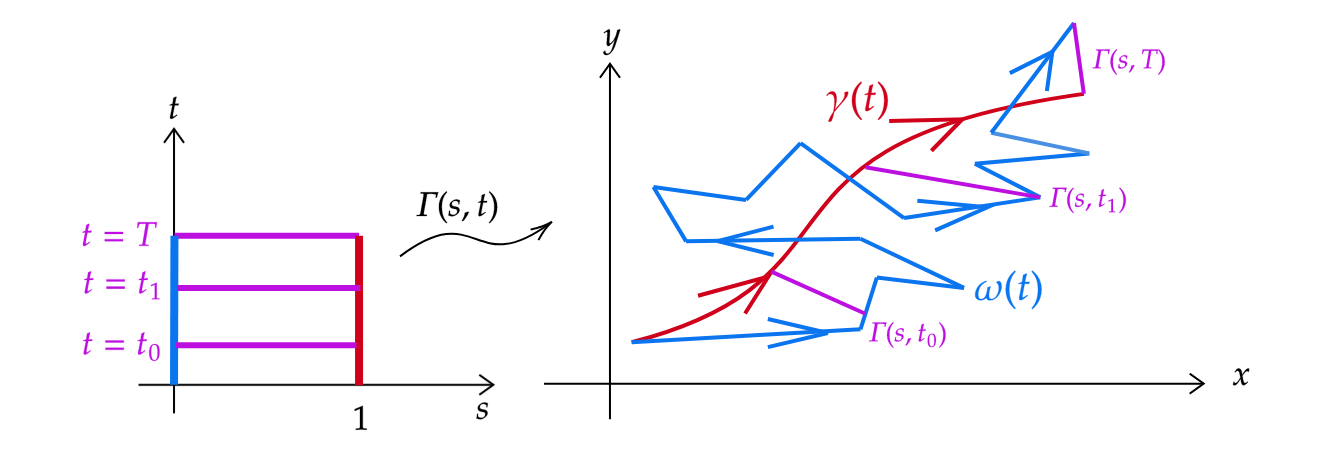}
    \caption{The Homotopy $\Gamma(s,t)$ between a fixed path $\gamma(t)$ and a Brownian $\omega(t)$. Purple denotes the straight-lines connecting $\gamma(t_0)$ to $\omega(t_0)$; $\bar\beta(\gamma(t_0))$ is defined as an integral of $\beta$ along such lines.}
\end{figure}

Given the function $\bar\beta$, we now define a (family of) magnetic Agmon metrics: for every $a>0$, define
\begin{equation}\label{eq:mag-agmon-weight}
    \rho_{\beta,\lambda}^a(x,y) = \inf_{\gamma:x\mapsto y} \int_0^1 \max\left(\sqrt{\bar\beta(\gamma(t))-2(\lambda-\nu_1/a^2}) \, ,0\right) \, |\dot\gamma(t)| \, dt,
\end{equation}
where $\nu_1$ is the ground state energy of $-\tfrac12\Delta$ on the unit disc. For ease of exposition we will often set $a=1$.

Our results hold for a broad class of potentials $A$, and are relatively agnostic to the fine structure of the associated $\beta$ (notice in particular we make no assumption about $\beta$ having a fixed sign). For more on this specific function class see \S \ref{sec:stochastic}.

\begin{thm}\label{thm:main}
    Assume the magnetic potential $A$ lies in the local vectorial Kato class, $A\in \vec{\cK}_{\loc}(\bbR^2)$, with magnetic field $\beta=\emph{curl}\, A$. Let $H(A)f=\lambda f$ be a Dirichlet eigenfunction, and define $E_\lambda=\{x\in \Omega: \tfrac12|A|\leq \lambda\}$ the classically allowed region at energy $\lambda$. Then, there exists $c_{a}>0$ depending only on $a$, with
    \[ |f(x)| \leq c_{a}\,  e^{-\rho_{\beta,\lambda}^a(x,E_\lambda)} \, \|f\|_\infty.  \] 
\end{thm}

We note here that because $\cC^1(\bbR^d,\bbR^d)\subset \vec{\cK}_{\loc}(\bbR^d)$, most vector potentials $A$ of physical interest are contained in this class. In particular we make no other assumptions on the structure of the magnetic field $\beta$ such as being e.g. radial, of fixed sign, possessing a unique or finitely many non-degenerate minima, etc. 

\subsection{An example}
We now consider theorem \eqref{thm:main} in the case of a uniform magnetic field. If $\beta\equiv \beta_0$ is constant, we can simplify our quantities (see \S \ref{sec:phys-levy} for more details) to find
\[ \bar\beta(\gamma(t)) = \frac{\beta_0^2}{4}\bbE[|\gamma(s)-\omega(s)|^2], \quad E_{\lambda_0} = \{\tfrac12|A|^2 \leq \lambda_0 \} = \{|y|^2\leq \beta_0\}. \]
With these simpler expressions for $\bar\beta(\gamma(t))$ and $E_{\lambda_0}$, we can prove that for $a=1$ and for any path $\gamma$ beginning at $x$ that its corresponding magnetic Agmon length is at least
\[ \int_0^1 \sqrt{(\bar\beta(\gamma(t)) - 2(\lambda-\nu_1))_+} \, |\dot\gamma(t)|dt  \geq \tfrac{\beta_0}{8}\dist(x,y)^2 , \] 
where $y\in \pa E_{\lambda_0}$ is the point achieving $\rho_{\beta,\lambda_0}(x,E_{\lambda_0})$. Infimizing over all paths from $x$ to the classical region, Theorem \eqref{thm:main} predicts that the ground state eigenfunction satisfies for all $a>0$ that
\begin{equation}\label{eq:inf-achieved}
   |f_0(x)| \leq c_{a} \exp\left\{-\tfrac{\beta_0}{8}\dist(x,y)^2  \right\} ,   
\end{equation} 
which recovers the leading order asymptotics of the true exponential decay of the ground state in this example, namely that
\[ f_0(x) = \tfrac{1}{\pi} \exp\{-\tfrac{\beta_0}{2}|x|^2\} .  \]

\subsection{Other consquences of Theorem \ref{thm:main}}
Along similar lines, if we assume more about the structure of the magnetic field we can conclude stronger statements about the rate of decay of eigenfunctions.

\begin{cor}\label{cor:confine}
    Assume $\beta$ is a confining potential, i.e. outside some compact set $\beta(x)\geq \beta_0|x|^2$. Then for any magnetic potential $A\in \cK_{\loc}(\bbR^2)$ with associated field $\beta$, any eigenfunction $f$ of $H(A)$ with eigenvalue $\lambda$ we have
    \[ |f(x)| \leq c_a \exp\left\{ -\tfrac{\beta_0}{6}\left(|x+y|^2-2\right) \right\} \|f\|_\infty \] 
    where $\rho_{\beta}^a(x,y) = \rho_\beta^a(x,E_\lambda)$.
\end{cor}
We can also consider cases where the magnetic field may not have a unique non-degenerate minima, e.g. if $\beta$ is concave. 
\begin{cor}\label{cor:concave}
    Assume $\beta\in \cC^0(\Omega)$ is a concave function. Then for any magnetic potential $A\in \cK_{\loc}(\bbR^2)$ with associated field $\beta$, we have for all $a>0$
    \[ |f(x)|\leq c_a\exp\left\{ \inf_{\gamma:x\mapsto y} -\int_0^1 \sqrt{ \tfrac{(\beta(\gamma(t)) - \inf \beta)^2}{36}(2t+ (|x|-|\gamma(t)|)^2) - 2(\lambda - \nu_1/a^2) } \, |\dot\gamma(t)| dt \right\} \|f\|_\infty \] 
\end{cor}

We can also prove a `Carmona-type' pointwise estimate for magnetic eigenfunctions valid on $\bbR^2$ rather than simply a bounded domain $\Omega$, which is similar to those established in \cite{carmona1978pointwise}. Assume the magnetic field satisfies possesses a decomposition 
\[ \text{$\beta = W - U$ where $W\in L_{\loc}^p(\bbR^2)$ and $U\in L^\infty(\bbR^2)$, for $p>1$ and $U\geq 0$} \] 
and satisfying $W_\infty:=\inf_{y\in \bbR^2} W(y)>-\infty$. Similar hypotheses were used in \cite{carmona1978pointwise} to establish pointwise decay for Schr\"odinger eigenfunctions with a scalar potential $V$; as in the scalar case, this splitting allows for magnetic fields $\beta$ which are not of fixed sign, but are not so singular that we cannot apply the Feynman-Kac-It\'o formula to our corresponding operator $H(A)$. In particular, an application of Khasminskii's Lemma (see e.g. \cite[Prop 4.105]{lorinczi2011feynman}) implies that if $f\in L^p(\bbR^d)+L^\infty(\bbR^d)$ for $p>d/2$ then $f\in \cK(\bbR^d)$, as required.

\begin{prop}\label{thm:carmona-bdd}
Let $\beta$ be be as above, and fix $a>0$. Let $y\in \bbR^2$ be a point satisfying $W(y)<W_\infty+a$. Then for any $T>0$ and every path $\gamma:[0,T]\to \bbR^2$ from $x$ to $y$ there exist constants $C_1,C_{2,a}>0$ such that
\[ |f(x)| \leq C_1  e^{\lambda T} \left( C_{2,a} e^{-\tfrac{a^2T}{2}(W_\infty-\|U\|_{L^\infty})^2} e^{-\tfrac{\nu_1T}{a^2}-\frac{\dist(x,y)^2}{2T}} + e^{-\tfrac{T}{2}W_a^\gamma(x)}  \right) \|f\|_{L^\infty} \] 
where $W_a^\gamma(x) := \inf\{ W(y) :  |\gamma(t)-y|\leq a, t\in [0,T], \gamma(0)=x \}$ is the infimum of $W$ over an $a$-neighborhood of the curve $\gamma$.
\end{prop}
In this splitting, the $L^p$ part of $\beta$, $W(x)$, is of indefinite sign, and thus can contribute to the most decay. By contrast the $L^\infty$ part, $U$, plays the role of ``worst possible growth" of the field $\beta$ on $\bbR^2$. It is the balancing of these two terms which gives this global decay estimate. In contrast to the usual Carmona-type estimate \cite[Lemma 3.1]{carmona1978pointwise}, there is more dependence on a path $\gamma$ from the starting point to one approaching $\inf W$. This path dependence also arises through the function $W_a^\gamma(x)$ in the statement denotes a function of both the path $\gamma$ and the $L^p$-part of $\beta$; this quantity represents the smallest value $W(y)$ approaches on an $a$-neighborhood of the path $\gamma$.  

\subsection{Related work}

In the `electric case' of a scalar potential $V$, the original Agmon metric was first introduced by the pioneering work Agmon in \cite{agmon}, and was prefigured by earlier work of Lithner \cite{lithner1964theorem}. The probabilistic approach to eigenfunction estimates which we employ in this paper was first popularized for the case of scalar potentials by Carmona \cite{carmona1978pointwise}, and rapidly applied to a study of the Agmon metric by Carmona \& Simon \cite{carmona-simon}. The latter work uses these methods to deduce that Agmon's estimate produces an essentially sharp lower bound for the ground state; the main reason why the method is restricted to the ground state is that one wants to avoid possible cancellation in the path integral and the ground state does not change sign. Already this analysis illustrates new subtleties that the purely magnetic case presents: because $H(A)$ eigenfunctions are $\bbC$-valued one cannot make such conclusions about the ground state having fixed sign.

As mentioned previously there is an enormous body of work analyzing the spectral theory of magnetic Schr\"odinger operators from various perspectives. As early as the work of Avron-Herbst-Simon \cite{avron1978schrodinger}, there have been indications of the role the magnetic field plays in localization. Namely they proved for an $L^2$-normalized magnetic eigenpair $H(A)\psi_n=\lambda_n\psi_n$ that $\int_\Omega \beta(x)|\psi_n(x)|^2dx\leq \lambda_n$, which implies that most of the $L^2$-mass of an eigenfunction should be contained in $\{ \beta(x)\leq \lambda_n \}$, provided that $\beta\geq 0$. Proceeding along these lines, one can expect an Agmon estimate in 2 dimensions with respect to $(B(x)-\lambda_n)_+$ for such non-negative fields, and this fact appears to be folklore among experts. In contrast, in the present work we make no assumption about the sign of the magnetic field.

Later works deduced an Agmon type estimate for the operator $L_{A,V}=H(A)+V$ are those of Helffer-Nourrigat \cite{helffer-nourrigat}; they assume that $A,V$ are polynomials, and combine this hypothesis with nilpotent Lie group techniques to derive their conclusions. Later works including \cite{brummelhuis1991exponential, fournais-helffer, helffer-morame, helffer1988effet} use semiclassical analysis to construct local model operators, by either exploiting assumptions about the structure of $A$ or $\beta$, as well as the geometry of $\Omega$, in order to construct normal forms for $H(A)$ (resp. $L_{A,V}$) and from these derive corresponding asymptotic approximations (e.g. WKB constructions) of their corresponding eigenfunctions. Along different lines, the work of \cite{poggi, poggi-mayboroda} proceeds by generalizing assumptions of \cite{shen1998bounds} made in the 2d case, which assume that the magnetic field $\beta$ has a ``strongly favored direction"; this hypothesis is combined with harmonic analysis techniques to obtain a variant of the Fefferman-Phong type uncertainty principle \cite{fefferman1983uncertainty, shen1996eigenvalue} with respect to a modified Filoche-Mayboroda landscape function \cite{filoche2012universal}, and from this conclude in turn that the fundamental solution decays exponentially. For both these works and the semiclassical approaches mentioned previously, Agmon-type decay for $L_{A,V}$ is often a consequence of the scalar potential $V$ dominating the magnetic field $\beta$ in some way, and reducing to the well-known scalar Agmon estimates for $V$.

In recent years, in order to directly confront the effects of the magnetic potential $A$, there has been a growing interest in the analysis of the purely magnetic Schr\"odinger operator $H(A)$. Most relevant and most recent are the results \cite{alfa2022tunneling, bonthonneau2021exponential, bonnaillie2022purely, helffer2022quantum, Fournais-ch, fournais2023purely} which appear as some of the few prior known results providing magnetic Agmon type localization estimates. These results all follow along similar lines as the semiclassical approaches described above, in exploiting hypotheses about the fine structure of the field $\beta$ and the geometry of the domain $\Omega$. 

\section{Background and probabilistic preliminaries}

We begin first with an informal discussion on the role of gauge invariance in the analysis of $H(A)=\tfrac12(-i\nabla -A(x))^2$, and the role it suggests should be played by the magnetic flux of $\beta=\text{curl }A$. Afterwards we fix  notation and give a short compilation of the function classes of vector potentials $A$ which we will consider, and recall the Feynman-Kac-It\'o formula Thm \ref{thm:broderix} used throughout.

\subsection{Gauge invariance and Green's theorem}\label{sec:gauge}

Let $\Omega\subset \bbR^2$ be a bounded domain with smooth boundary, and impose Dirichlet boundary conditions on the magnetic Laplacian $H(A)$, associated to a vector field $A:\Omega \to \bbR^2$. This operator has, under mild regularity assumptions on $A$, a spectral resolution by countably many eigenpairs $(f_j,E_j)_{j=0}^\infty$ corresponding to purely discrete spectrum, with $f_j: \Omega \to \bbC$. 

A fundamental property possessed by $H(A)$ is a notion of gauge-invariance with respect to modulating the phase of an eigenfunction: for any choice of $\varphi\in \CI(\Omega,\bbR)$ we have
\[ e^{i\varphi}H(A)e^{-i\varphi} = H(A-\nabla \varphi), \] 
hence the spectrum is invariant with respect to this change, while eigenfunctions are modified only by some (possibly non-trivial phase): $f\mapsto e^{-i\varphi} f$. In particular the pointwise norm $|f(x)|$ of $H(A)$ eigenfunctions are preserved by this gauge change.

Now given a domain $\Omega$ that is simply-connected, if the magnetic field $\beta=\pa_1A_2-\pa_2A_1$ is identically zero throughout $\Omega$ then we must have that $A=\nabla \varphi_0$ for some $\varphi_0\in \CI(\Omega)$. Thus $H(A)$ is gauge equivalent to $-\tfrac12\Delta$, and hence the eigenfunctions of $H(A)$ satisfy the same pointwise bounds as those of $-\tfrac12\Delta$ (with the given boundary conditions). On the other hand, what if $\beta\neq 0$? We now turn to the heat kernel to say more in this case of non-trivial magnetic field $\beta$. \\

On all of $\bbR^d$ we have an explicit formula \cite[Eqn. 4.4.8]{lorinczi2011feynman}, for the magnetic heat kernel semigroup. Namely for $t>0$, $e^{-tH(A)}:L^2(\bbR^d)\to L^2(\bbR^d)$ has integral kernel given as
\begin{equation}\label{eq:mag-heat-kernel}
    \kappa_t(x,y) = (2\pi t)^{-d/2}\exp\left( \frac{-|x-y|^2}{2t} \right) \exp\left( -\frac{i}{2}[A(x)+A(y)]\cdot (y-x) \right) . 
\end{equation} 
The first term is the celebrated Euclidean heat kernel, and its appearance is unsurprising; more interesting is the second quantity which plays the role of an ``infinitesimal parallel transport". Namely, the relative phase in this kernel function at a pair of points can be expressed as the imaginary exponential of a real line integral along the straight line path $\mathbf{r}(t)=x+t(y-x)$, 
\[ \frac{A(x)+A(y)}{2}\cdot (y-x) \approx \int_0^1 A(\mathbf{r}(t))\cdot \mathbf{r}^\prime(t)dt = \int_{\mathbf{r}}A\cdot d\mathbf{s}, \] 
which suggests this new quantity in our magnetic heat kernel is roughly of the form $\exp\left(-i\int_x^y A\cdot d\mathbf{s}\right)$. We have approximated this line integral by the average of its endpoints, which is more accurate for points $x$ and $y$ which are closer. Points which are further away can contribute possibly opposing phases, but must also be exponentially damped by the Euclidean heat kernel factor in $\kappa_t(x,y)$.

However, this line integral of $A$ was computed along the straight line path connecting $x$ and $y$. If $\gamma$ is any other $C^1$ path connecting these points then Green's theorem implies
\begin{equation}\label{eq:greens-thm}
    \int_{\gamma}A\cdot d\mathbf{s} = \int_{\mathbf{r}} A\cdot d\mathbf{s} + \iint_{D(\mathbf{r},\gamma)} \beta \, dxdy 
\end{equation}  
where $D(\mathbf{r},\gamma)\subset \bbR^2$ is the positively oriented domain enclosed by the union of $\gamma$ and $-\mathbf{r}$ (the straight line path with opposite orientation). This is the second appearance of the magnetic field $\beta$, and again suggests why the norm of the magnetic field should dominate eigenfunction behavior. Namely, if the shortest path $\mathbf{r}(t)$ connecting $x$ to $y$ passes through or near regions where $\{|\beta|\gg 0\}$ then distinct paths can contribute very different and possibly destructively interfering phases to $\kappa_t(x,y)$.

\begin{figure}[h]%
    \centering
    {{\includegraphics[width=4cm]{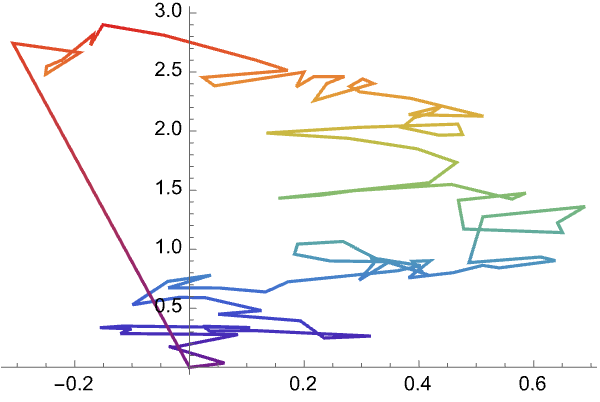} }}%
    \qquad
    {{\includegraphics[width=4cm]{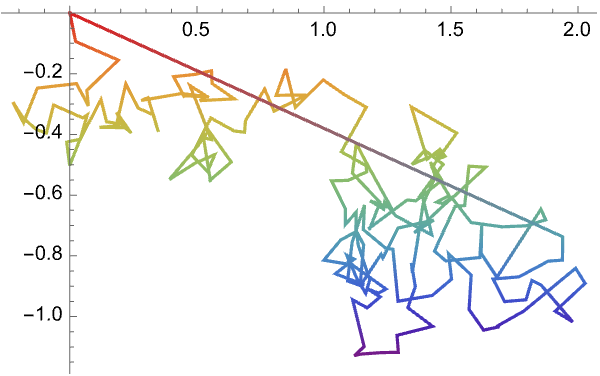} }}%
    \caption{Boundaries of two different domains formed by closing a planar Brownian motion $\omega$ with a straight line between its start and endpoints. The $\beta$ integral of \eqref{eq:greens-thm} represents the enclosed area of such regions when $\mathbf{r}\mapsto \omega$}%
    \label{fig:1}%
\end{figure}

The previous discussion suggests that for an arbitrary potential $A$ we should consider the contribution of the ``average magnetic flux"  along all paths $\gamma$ homotopic to the shortest path $\mathbf{r}$. First we express the domain $D(\mathbf{r},\gamma)$ via a homotopy between the two paths, 
\[ \Gamma:[0,1]\times [0,T]\to D(\gamma,\mathbf{r})\subset \Omega, \quad \Gamma(s,t) = s\mathbf{r}(t)+(1-s)\gamma(t), \] 
so by a change of variables the integral can also be expressed in terms of this homotopy,
\begin{align*}
    \iint_{D(\gamma,\mathbf{r})} \beta \, dx &= \int_0^1 \int_0^1 \beta(\Gamma(s,t)) \, |\text{det}D\Gamma|(s,t)\,dsdt \\
    & = \int_0^1 \int_0^1 \beta(\Gamma(s,t)) \, [\pa_s \Gamma(s,t)\times \pa_t \Gamma(s,t)] \, dsdt , 
\end{align*} 
where we denote $\mathbf{v}\times \mathbf{w} = \mathbf{v}\cdot J\mathbf{w} := v_1w_2-v_2w_1$. 

So far the discussion has been purely formal; to convert these physical heuristics into a rigorous argument, we need to make more precise in what sense we can ``average over all domains determined by paths homotopic to $\gamma$". The proper framework for this is by reinterpreting the quantities in \eqref{eq:greens-thm} as stochastic integrals; this can be done via the Feynman-Kac-It\'o formula for the magnetic heat kernel \ref{thm:broderix}, which we will introduce after some probabilistic preliminaries.

\subsection{Preliminaries on stochastic line integrals}\label{sec:stochastic}

We now begin gathering some of the crucial definitions for our results.

\begin{definition}
A vector potential $A:\Omega\to \bbR^d$ is said to lie in the vectorial Kato class $A\in \vec{\cK}(\bbR^d)$ if its squared norm $|A|^2$ and divergence $\Div A$, considered as a distribution on $\CI_0(\bbR^d)$, are both elements of the Kato class $\cK(\bbR^2)$. A scalar function $V\in \cK(\bbR^d)$ lies in the Kato class if,
\[ \lim_{r\to 0+} \sup_{x\in \bbR^2} \int_{\bbR^d} g_r(x-y)|V(y)| dy=0  \]
where $g_r$ is the Coulomb potential outside the ball $B_r(0)$, 
\[ g_r(x):=\chi_{B_r(0)}\begin{cases}
    -\ln|x|, & d=2 \\
    |x|^{2-d} & d\geq 3
\end{cases} \; . \]
In both cases a vector potential $A$ or scalar function $V$ belongs to their local Kato classes $\vec{\cK}_{\loc}(\bbR^d)$, $\cK_{\loc}(\bbR^d)$ if $A\chi_K$, $V\chi_K$ are Kato for every compact $K\subset \bbR^d$.
\end{definition}

There is an alternate definition of these function spaces, as a result of \cite[Thm 4.5]{aizenman1982brownian}, which states that the Kato class can be characterized via Brownian motion,
\begin{equation}\label{eq:kato-crit}
    V\in \cK(\bbR^d) \iff \lim_{t\to 0^+} \sup_{x\in \bbR^d}\bbE\left[ \int_0^t |V(\omega_x(t'))|dt' \right] = 0 . 
\end{equation}  
A different but similarly useful characterization of $\cK(\bbR^d)$ is in terms of the spaces $L_{\text{unif,loc}}^p(\bbR^d)$; for a (possibly unbounded) open subset $\Omega\subset \bbR^d$ we define
\[ ||V||_{L_{\text{unif,loc}}^p(\Omega)} := \begin{cases}
    \sup\limits_{x\in \Omega} \left(\int_{\Omega} |V(x)|^p \, \chi_{B_1(x-y)} dy \right)^{1/p} & 1\leq p<\infty \\
    \esssup\limits_{x\in \Omega} |V(x)| & p=\infty . 
\end{cases}  \] 
Having defined these uniformly locally $p$th-power integrable functions we state the following useful inclusions, for $p>\tfrac{d}{2}$ we have
\[ L_{\text{unif,loc}}^p(\bbR^d) \subset \cK(\bbR^d) \subset L_{\text{unif,loc}}^1(\bbR^d)  \] 
\[ L_{\text{loc}}^p(\bbR^d) \subset \cK_{\loc}(\bbR^d) \subset L_{\text{loc}}^1(\bbR^d) , \] 
which may be more useful than the criterion \eqref{eq:kato-crit}.

With these preliminaries we can now recall 
\begin{thm}[Prop 2.3\cite{broderix2000continuity}]\label{thm:broderix}
Let $A\in \vec{\cK}_{\loc}(\bbR^d)$, and $\Omega\subset \bbR^d$ open. Then for any $\psi\in L^2(\Omega)$ and $t\geq 0$, one has that the Dirichlet magnetic heat kernel on $\Omega$ satisfies
    \[ e^{-tH(A)}\psi(x) = \bbE\left[ \psi(\omega(t)) e^{i\int_0^t A(\omega(t'))\cdot d\omega(t') - \tfrac{i}{2}\int_0^t \Div A(\omega(t)) dt }\;\Xi_{\Omega,t}(\omega) \right] \]
where $\omega(t)$ is a Brownian motion started at $\omega(0)=x$, and 
    \[  \Xi_{\Omega,t}(\omega):=\begin{cases} 1 & \omega(t')\in \Omega, \,  0\leq t'\leq t \\
        0 & \text{else}
    \end{cases} \]
\end{thm}

\subsection{Basic idea of the proof} The connection between pointwise eigenfunction estimates and the stochastic integrals studied arises from the Feynman-Kac-It\'o formula: if $f$ is an eigenfunction of $H(A)$ then it satisfies $e^{-tH(A)}f(x)=e^{-t\lambda}$, which combined with Thm \ref{thm:broderix} implies
\[ f(x) = e^{t\lambda} \, \bbE\left[ f(\omega(t)) e^{i\int_0^t A(\omega(t'))\cdot d\omega(t') - \tfrac{i}{2}\int_0^t \Div A(\omega(t)) dt } \right] ,  \] 
thus estimates of this expectation become estimates of the eigenfunction.

After gauge-fixing such that $\Div A\equiv 0$, the object of interest for us is ultimately
\begin{equation}\label{eq:char-fcn}
     \bbE [e^{i\int_0^t A(\omega)\cdot d\omega}]. 
\end{equation}
The quantity in the exponential is an It\'o integral, which is up to a time-rescaling, representable as Brownian motion, i.e. 
\[ \forall t, \quad \int_0^t A(\omega_x(t'))\cdot d\omega_x(t') \sim N\left(0,\int_0^t \bbE\left[|A|^2(\omega_x(t'))\right] \, dt'\right) . \] 
So this stochastic line integral represents a process which for all times $t>0$ is normally distributed with mean zero and variance $\int_0^t \bbE\left[|A|^2(\omega_x(t'))\right] \, dt'$. In other words, the function in \eqref{eq:char-fcn} is, for each $t$, precisely the characteristic function of a normal distribution. For an arbitrary normally distributed random variable $N(\mu,\sigma^2)$ this has the well-known formula
\[ \bbE[e^{i\zeta N(\mu,\sigma^2)}] = e^{i\zeta\mu - \tfrac12 \zeta^2\sigma^2} .  \]
In the case of the It\'o integral arising from our potential $A$, this has no expectation and it is the variance which drives exponential decay. This idea is at the core of Thm \eqref{thm:main}; we get yet better exponential decay by first replacing the integral of $A$ with the stochastic analogue of \eqref{eq:greens-thm} to get exponential decay depending on stochastic double integrals the magnetic field $\beta$. 

\section{Uniform Magnetic fields and L\'evy's stochastic area formula}\label{sec:phys-levy}

In this section we explain further an example from the introduction, the magnetic Sch\"odinger operator with a constant magnetic field $\beta(x) \equiv \beta_0\in \bbR_{>0}$. In addition to confirming the calculation previously presented, we can also use this example to highlight the role played by $\bar\beta(\gamma(t))$, and relate it to the celebrated L\'evy stochastic area formula.

If $\beta(x)\equiv \beta_0$, by gauge invariance we are free to assume that our magnetic Schr\"odinger operator $H(A)$ derives from a specific magnetic potential, the so-called Landau gauge, $A=\frac{\beta_0}{2}(-x_2,x_1)$. The corresponding magnetic Schr\"odinger operator is of the form
\[ H(A_0) = -\tfrac12 \Delta - i \tfrac{\beta_0}{4} (x_1\pa_2 - x_2\pa_1) + \tfrac{\beta_0^2}{8}|x|^2, \]
and has explicitly computable eigenvalues/eigenfunctions. We consider just the ground state eigenfunction $f_0$, which has eigenvalue $\tfrac{\beta_0}{2}$, and the form,
\[ f_0 = \tfrac{1}{\pi} e^{-\frac{\beta_0}{2}|x|^2}. \] 
Because $\beta$ is constant we can write $\bar\beta(\gamma(t)) = \tfrac{\beta_0^2}{4}\bbE[|\gamma(t)-\omega(t)|^2]$. Further, since $\bbE[|\omega(t)|^2] = |x|^2 + 2t$, and $\bbE[\omega(t)]=x$ we have
\begin{align*}
   \bbE[|\gamma(t)-\omega(t)|^2] & = |\gamma(t)|^2 - 2\langle \gamma(t), \bbE[\omega(t)]\rangle +\bbE[|\omega(t)|^2] \\
   & \geq |\gamma(t)|^2 - 2 |\gamma(t)|\, |x|  + |x|^2+ 2t  = (|x|-|\gamma(t)|)^2 + 2t
\end{align*}   
Now we introduce two relevant times for the path: because $\omega(0)=\gamma(0)=x$, the quantity $\bar\beta(\gamma(t))$ is close to zero near $t=0$, so not all of the integration path is expected to contribute to the integral. We define the exit times,
\[ \text{  $t_F=\sup_{t>0}\left\{ \big||x|-|\gamma(t)|\big| \leq \frac{\dist(x,y)}{2}\right\}$ and $t_0 = \sup_{t>0} \left\{ t \leq \frac{2(\beta_0-2\nu_1)}{\beta_0^2} \right\}$  } \]
and we set $\tau=\max\{t_0,t_F\}$, the first time where positivity of the integrand is guaranteed on the subinterval $[\tau,1]$. Namely
\begin{align*}
    \rho_{\beta,\lambda}(x,E_\lambda) & \geq \int_0^1 \sqrt{\left(\tfrac{\beta_0^2}{4}\bbE[|\gamma(t)-\omega(t)|^2] - (\beta_0-2\nu_1) \right)_+ } \, |\dot\gamma(t)|\, dt  \\
    & \geq \int_{\tau}^1 \sqrt{\tfrac{\beta_0^2}{4}(|x|-|\gamma(t)|)^2+2t) - (\beta_0-2\nu_1)} \, |\dot\gamma(t)| \, dt \\
    & \geq \tfrac{\beta_0}{2}\int_{\tau}^1 \big||x|-|\gamma(t)|\big| \,|\dot\gamma(t)|dt ;  
\end{align*}  
In the second inequality we have used that the integrand is strictly positive since for all time $t\in [\tau,1]$ since $t\geq\tfrac{2(\beta_0-2\nu_1)}{\beta_0^2}$, using $\tfrac{2(\beta_0-2\nu_1)}{\beta_0^2}<.12$ since $2\nu_1\approx 4.8$.  

The inequality that $\rho_{\beta}^1(x,E_\lambda)\geq \tfrac{\beta_0}{8}\dist(x,E_\lambda)^2$ will finally follow from the mean value theorem: there exists $t^*\in [\tau,1]$ satisfying
\begin{align*}
    \rho_{\beta,\lambda}(x,E_\lambda) &\geq \tfrac{\beta_0}{2}\int_{\tau}^1 \big||x|-|\gamma(t)|\big| \,|\dot\gamma(t)|dt \\
    & =\tfrac{\beta_0}{2} \big||x|-|\gamma(t^*)|\big| \int_{\tau}^1 |\dot\gamma(t)| dt  \\
    & \geq \tfrac{\beta_0}{2} \tfrac{\dist(x,y)}{2}\text{len}(\gamma|_{[\tau,1]})  \\
    & \geq \tfrac{\beta_0}{2} \left(\tfrac{\dist(x,y)}{2}\right)^2 
\end{align*}
since by the definition of $t_F\leq \tau$, $\gamma(\tau)$ and $y=\gamma(1)$ must have distance at least $\tfrac{\dist(x,y)}{2}$. This completes the claim establishing \eqref{eq:inf-achieved}, thus by Theorem \eqref{thm:main} implying
\[ |f_0(x)|\leq C_1 \|f\|_\infty \exp\left\{-\tfrac{\beta_0}{8}\dist(x,y)^2\right\} . \]

Having exhibited the theorem in this case, we can now discuss a connection with a classical result in probability and the heat kernel of a magnetic Schr\"odinger operator with uniform field strength. For any $A\in \vec{\cK}(\bbR^2)$ we have that the action of the heat kernel arises, via the Feynman-Kac-It\'o formula \ref{thm:broderix}
\begin{equation}\label{eq:e-val-eqn}
    e^{-tH(A)}f(x) = \bbE\left[ f(\omega_t)e^{i\int_0^t A(\omega_s)\cdot d\omega_s - \tfrac{i}{2}\int_0^t \Div A(\omega_s)ds} \right] . 
\end{equation} 
On the other hand, for the case of a uniform magnetic field $\beta=\beta_0$ where we can gauge-fix to the divergence free potential $A(x) = \tfrac{\beta_0}{2}(-x_2,x_1)$, in which case the quantity in the exponential becomes
\[ i\int_0^t A(\omega_s)\cdot d\omega_s = i \frac{\beta_0}{2}\int_0^t \omega_2(s)d\omega_1(s) - \omega_1(s)d\omega_2(s) =:i\tfrac{\beta_0}{2}\cA_t. \] 
This latter quantity is the celebrated L\'evy stochastic area \cite{levy1951wiener} for planar Brownian motion, which represents the expected area of a domain formed by enclosing Brownian motion with a straight-line.

Now we can study the quantity $\bbE[\exp(i\beta_0 \cA_t)]$, which satisfies by \cite[eqn 3.12]{ikeda1995levy},
\begin{align}\label{eq:Levy-area-formula}
    \bbE_{0,x}^{t,y}[e^{i\beta_0 \cA_t}] & = \frac{\beta_0}{4\pi \, \sinh(\tfrac{\beta_0 t}{2})} \exp\left( - \tfrac{\beta_0}{2} \coth(\tfrac{\beta_0 t}{2})\tfrac{|x-y|^2}{2} + \tfrac{i\beta_0}{2}(x_2y_1-x_1y_2) \right)  ,
\end{align} 
where we have conditioned on Brownian paths $\omega(0)=x$, and $\omega(t)=y$. We now observe that this matches exactly with the celebrated Mehler kernel \cite[pg. 168]{simon1979functional}, the integral kernel of $e^{-tH(A)}$ for the given potential $A(x)$.

\section{Proofs}

We can now quickly present the proofs of our main results.
\begin{proof}[Proof of Thm \ref{thm:main}]
Let $H(A)=\tfrac(-\nabla-iA)^2$ and $H(A)f=\lambda f$ be an eigenfunction, and fix $x\in \Omega^\circ$. Let $T_{\Omega,x}>0$ be the exit time for Brownian motion started at $x$, and for all $T'>0$ we set $T=\min\{T',T_{\Omega,x}\}$. By the Feynman-Kac-It\'o formula, theorem \ref{thm:broderix}, we have 
\begin{equation}\label{main-eq}
    |f(x)| = e^{\lambda T}\left| \bbE \left[ f(\omega(T)) e^{i\int_0^T A(\omega(t))\circ d\omega(t)} \right] \right| \leq \|f\|_{\infty} \left| \bbE \left[e^{\lambda T+ i\int_0^T A(\omega(t))\circ d\omega(t)} \right] \right| 
\end{equation}
and the last term on the right can be rewritten using Ikeda-Manabe's Stochastic Stokes formula \cite[Thm 7.1]{ikeda-manabe}: for any fixed $C^1$ path $\gamma:[0,T]\to \Omega$ let $\Gamma(s,t)=s\omega(t)+(1-s)\gamma(t)$ be the straight-line homotopy connecting $\omega(t)$ to $\gamma(t)$. Then 
\begin{align}\label{eqn:IM-formula}
    \int_0^T & A(\omega(t))\circ d\omega(t)  - \int_\gamma A  = \iint_{D(\gamma,B)}\beta 
\end{align}
where the process $\iint_{D(\gamma,B)}\beta$ on the right hand side is defined to be
\begin{align*}
    \iint_{D(\gamma,\omega)}\beta & =  \int_0^T \left(\int_0^1  s \,  \beta(\Gamma(s,t)) ds \right)\big[ (\omega_2(t)-\gamma_2(t))\circ d\omega_1(t) - (\omega_1(t) - \gamma_1(t) \circ d\omega_2(t)\big] \\
    & \quad\quad+ \int_0^T \left(\int_0^1  (1-s) \,  \beta(\Gamma(s,t))  ds \right) [\omega(t)-\gamma(t)]\wedge \dot\gamma(t) dt \\
    & =: \int_0^T \beta_1(\omega(t),t)\cdot d\omega(t) + \int_0^T \beta_2(\omega(t),t) dt
\end{align*}  
Notice that for $\omega(t)$ given by a deterministic $C^1$ path, the formula \eqref{eqn:IM-formula} reduces to the usual Green's theorem for the domain bounded between these curves. On the other hand, in our setting this process satisfies 
\[  \bbE \left[\iint_{D(\gamma,\omega)}\beta\right]  = \int_0^t \bbE[\beta_2(B_s,s)] ds, \]
and, 
\begin{align*}
     \text{Var} \left[\iint_{D(\gamma,\omega)}\beta\right] & = \int_0^T \bbE[|\beta_1(\omega(t),t)|^2] dt \\
     & = \int_0^T \bbE\left[\left|  \left(\int_0^1 s \beta(\Gamma(s,t)) ds \right) ((\omega_2(t)-\gamma_2(t)), -(\omega_1(t)-\gamma_1(t)))\right|^2 \right] dt \\
     & = \int_0^T \bbE\left[\left(\int_0^1 s\,\beta(\Gamma(s,t))ds\right)^2 |\omega(t)-\gamma(t)|^2\right] dt 
\end{align*}
The reason for considering these quantities is that since $\iint_{D(\gamma,B)}\beta$ is equivalent to a time-changed Brownian motion, with the above expectation and variance, thus we can compute
\[ \bbE[e^{i\iint_{D(\gamma,\omega)}\beta}] = \exp\left\{i \int_0^T \bbE[\beta_2(B_t,t)] dt - \frac{1}{2} \int_0^T \bbE\left[\left(\int_0^1 s\,\beta(\Gamma(s,t))ds\right)^2 |\omega(t)-\gamma(t)|^2\right] dt \right\}.  \] 
In particular this we have that 
\begin{align*}
    \left|\bbE [e^{\lambda T+ i\int_0^t A(\omega(t))\circ d\omega(t)}]\right| & = \exp\left\{ \lambda T - \frac{1}{2} \int_0^T \bbE\left[\left(\int_0^1 s\,\beta(\Gamma(s,t))ds\right)^2 |\omega(t)-\gamma(t)|^2\right] dt\right\}  \\
    & =: \exp\left\{ -\tfrac12\int_0^T \left[\bar\beta(\gamma(t)) - 2\lambda\right] dt \right\} 
\end{align*} 
And it is this final quantity which will play the role that a scalar potential $V(x)$ plays for the usual Agmon metric. This is because, by the main theorem of \cite{fujita1982onsager}, the probability that a Brownian path $\omega(t)$ stays within a fixed distance from $\gamma(t)$ can be estimated by
\[ \bbP( |\omega(t)-\gamma(t)|\leq a, \, 0\leq t\leq T) = C_2\exp\left(-\frac{\nu_1T}{a^2} -\tfrac12\int_0^T |\dot\gamma(t)|^2 dt+o_a(1)\right)  \] 
where $\nu_1$ is the first non-zero eigenvalue of the Dirichlet Laplacian on the unit disc, and $C_2=f_1(0)\int_{\bbD}f_1(x)dx$, where $f_1$ is the corresponding Dircihlet ground state eigenfunction of the unit disc. Now, returning to \eqref{main-eq} we have
\begin{align*}
    |f(x)| & \leq \|f\|_\infty \left|\bbE\left[e^{\lambda T +i\int_0^T A(\omega(t))\circ d\omega(t)}\bigg|\, |\omega(t)-\gamma(t)|\leq a, \, 0\leq t\leq T\right] \right| \\
    & \quad \quad \quad \quad \quad \times \bbP(|\omega(t)-\gamma(t)|\leq a,\, 0\leq t\leq T )  \\
    & \leq C_{a} \|f\|_\infty \exp\left\{ -\tfrac12\int_0^T |\dot\gamma(t)|^2ds -\tfrac12\int_0^T [\bar\beta(\gamma(t))-2(\lambda-\nu_1/a^2)]dt \right\} \\ 
    & \leq C_{a} \|f\|_\infty \exp\left\{ - \int_0^T \sqrt{(\bar\beta(\gamma(t))-2(\lambda-\nu_1/a^2))_+}\, |\dot\gamma(t)|dt \right\} 
\end{align*} 
where we have used in order, the original inequality of \eqref{main-eq}, the absolute value to remove the modulating factors (i.e. the $\beta_2^{\gamma}$ term), and Young's inequality $ab\leq \tfrac{a^2+b^2}{2}$ for the integrand. Ultimately this rests on the relation between the new weighted length functional $\rho_\beta$ and its associated the `classical action' $\cA_\beta$, 
\begin{align*}
   \cL_\beta(\gamma) &:=  \int_0^T \sqrt{(\bar\beta(\gamma(t))-2(\lambda-\nu_1/a^2))_+}\, |\dot\gamma(t)|dt \\
   &\leq \tfrac12\int_0^T |\dot\gamma(t)|^2 dt + \tfrac12 [\bar\beta(\gamma(t))-2(\lambda-\nu_1/a^2) ] dt =: \cA_\beta(\gamma) .  
\end{align*} 
Infimizing over the former gives the new Agmon distance, where the `classically allowed region' arises as $E_\lambda=\{|A|^2\leq 2\lambda\}$. In the final stage of this argument we use the fact that $\cL_\beta$ is independent of the parametrization. We can see this as if we take $\wt \gamma(t)=\gamma(\phi(t))$ then $t' = \phi^{-1}(t)$, $dt' = (\phi^{-1})'(t) dt$
\begin{align*}
     \int_0^{T'} \max\left(\sqrt{ \bar\beta(\wt\gamma(t)) - 2(\lambda-\nu_1/a^2) } ,0 \right)\, & |\wt\gamma^\prime(t)| \, dt  \\
     &\hspace{-2cm}= \int_0^{T} \max\left(\sqrt{ \bar\beta(\gamma(t)) - 2(\lambda-\nu_1/a^2) },0 \right) \, |\gamma^\prime(t)| \, (\phi^{-1})' dt \\
     &\hspace{-2cm} = \int_0^{T} \max\left(\sqrt{ \bar\beta(\gamma(t)) - 2(\lambda-\nu_1/a^2) } ,0 \right)\,|\gamma^\prime(t)| \, dt 
\end{align*}
as claimed.
\end{proof}

\begin{proof}[Proof of Cor \ref{cor:confine}]
Notice first that because $\bbE[XY]\geq \bbE[X]\bbE[Y]$, we have
\begin{align*}
    \bar\beta(\gamma(t)) &= \bbE\left[ \left(\int_0^1 s \beta(\Gamma(s,t)) ds\right)^2 |\omega(t)-\gamma(t)|^2\right] \\
    &\geq \bbE\left[\left(\int_0^1 s \beta(\Gamma(s,t)) ds\right)^2\right] \, \bbE[|\omega(t)-\gamma(t)|^2] \\
    &\geq \left(\bbE\left[\int_0^1 s \beta(\Gamma(s,t)) ds\right]\right)^2 \, \bbE[|\omega(t)-\gamma(t)|^2]
\end{align*} 
where we have applied Jensen's inequality in the final step. Now, given the assumption that $\beta(x)\geq \beta_0|x|^2$ outside a compact set, we have
\[\beta(\Gamma(s,t))\geq \beta_0(s^2|\omega(t)|^2+(1-s)^2|\gamma(t)|^2+2s(1-s)\langle \omega(t),\gamma(t)\rangle).\] 
By linearity of expectation and the facts that $\bbE[\omega(t)] = x$ and $\bbE[|\omega(t)|^2]=|x|^2+2t$ we compute,
\begin{align*}
    \bbE[s\beta(\Gamma(s,t))] & \geq 
    s\beta_0 \left( s^2 (|x|^2+2t) + (1-s)^2 |\gamma(t)|^2 + 2s(1-s)\langle x,\gamma(t)\rangle \right) \\
    & = s\beta_0 (2ts^2 + |sx + (1-s)\gamma(t)|^2). 
\end{align*} 
Combining these inequalities with the identity $\bbE[|\omega(t)-\gamma(t)|^2]=2t+(|x|-|\gamma(t)|)^2$ we can conclude,
\begin{align*}
\bar\beta(\gamma(t)) & \geq \left(\bbE\left[\int_0^1 s \beta(\Gamma(s,t)) ds\right]\right)^2 \, \bbE[|\omega(t)-\gamma(t)|^2] \\ 
    & \geq \beta_0^2\left(\int_0^1 2t\,s^3 + s|sx+(1-s)\gamma(t)|^2 \, ds \right)^2 \, (2t+(|\gamma(t)|-|x|)^2) \\
    & =  \beta_0^2 \left( t + \tfrac{|x|^2}{3} + \tfrac{|x+\gamma(t)|^2}{6} \right)^2  \, (\tfrac{t}{2} + \tfrac{(|\gamma(t)|-|x|)^2}{4}) \geq \beta_0^2 \left(\tfrac{|x|^2}{3}+\tfrac{|x+\gamma(t)|^2}{6}\right)^2 
\end{align*}
where the last inequality will hold for all $t>\tau=\sup_{t>0}\{\big||x|-|\gamma(t)|\big|\leq 2\}$. In fact, by choosing the transversal gauge 
\[ A(x,y) = \left(\int_0^1 \beta(t(x,y))\, dt\right)(-y,x), \] 
which satisfies $\Div A=0$ and $\text{curl }A=\beta$, then  $\gamma(t)\not\in E_\lambda := \{ |A(x,y)|^2\leq 2\lambda \}$ implies that for all $t\in [0,1]$,
\[ \left(\int_0^1 \beta(t\,\gamma(t))\, dt\right)^2|\gamma(t)|^2 >2\lambda .\] 
Thus 
\[ \frac{|x|^2}{3} \geq \frac{|\gamma(t)|^2}{3}\geq \frac{2\lambda}{3\left(\int_0^1 \beta(t\,\gamma(t))\, dt\right)^2} \geq \frac{2\lambda}{3\int_0^1 \beta^2(t\, \gamma(t)) dt} = \frac{2\lambda}{3\beta^2(t_0\, \gamma(t_0))} , \]
for all $t\in [0,1]$. Choosing $a>0$ sufficiently small that 
\[ \frac{\lambda}{3\min_{t\in [0,1]} \beta^2(t\, \gamma(t))}> (\lambda - \nu_1/a^2). \] 
Now because 
\[ \frac{d}{dt}|x+\gamma(t)|^2 = 2\langle \dot\gamma(t),x+\gamma(t)\rangle \leq 2|\dot\gamma(t)|\,|x+\gamma(t)|  \] 
we conclude
\begin{align*}
    \rho_{\beta,\lambda}^a(x,E_\lambda) & \geq \int_\tau^1 \sqrt{\left(\beta_0^2 \left(\tfrac{|x|^2}{3}+\tfrac{|x+\gamma(t)|^2}{6}\right)^2 - 2(\lambda-\nu_1/a^2)\right)_+ } \, |\dot\gamma(t)|\,dt \\
    & \geq \tfrac{\beta_0}{6}\int_\tau^1 |x+\gamma(t)|^2\, |\dot \gamma(t)| \, dt  \geq \tfrac{\beta_0}{6} \int_\tau^1 \tfrac12 |x+\gamma(t)|\, \frac{d}{dt}|x+\gamma(t)|^2 dt \\
    & \geq \tfrac{\beta_0}{6}\int_{\tau}^1 \frac{d}{dt}|x+\gamma(t)|^2 dt \\
    &\geq \tfrac{\beta_0}{6} (|x+y|^2 - 2), 
\end{align*}
where $y\in \partial E_\lambda$ is the point attaining $\rho_{\beta,\lambda}^a(x,E_\lambda)$, and we again take $\tau>0$ sufficiently large that $|x+\gamma(t)|^2>2$ on $[\tau,1]$. From this we observe 
\end{proof}

\begin{proof}[Proof of Cor \eqref{cor:concave}]
    Assuming that $\beta(x)$ is a concave function, we have that $\beta(\Gamma(s,t))\geq s\beta(\omega(t))+(1-s)\beta(\gamma(t))$. Arguing as above we have that
\begin{align*}
    \bar\beta(\gamma(t)) & \geq \bbE\left[ \left(\int_0^1 s\beta(\Gamma(s,t)) ds  \right)^2\right] \bbE[|\omega(t)-\gamma(t)|^2] \\
    & \geq \bbE\left[ \left( \int_0^1 s^2\beta(\omega(t)) + s(1-s)\beta(\gamma(t)) ds \right)^2 \right] ((|x|-|\gamma(t)|)^2+2t) \\ 
    & = \bbE\left[ \left( \frac{\beta(\omega(t)) + \beta(\gamma(t))}{6} \right)^2\right] ((|x|-|\gamma(t)|)^2+2t) \\
    & \geq \tfrac{1}{36}(\bbE[\beta(\omega(t))] + \beta(\gamma(t)))^2 ((|x|-|\gamma(t)|)^2+2t) \\
    & \geq \tfrac{(\beta(\gamma(t)) - \inf \beta)^2}{36}((|x|-|\gamma(t)|)^2+2t) 
\end{align*}
thus after applying Thm \eqref{thm:main} we conclude that for any $C^1$ path $\gamma(t)$ from $x$ to $y$, we have,
\[  |f(x)| \leq C_a\exp\left( -\int_0^1 \sqrt{ \tfrac{(\beta(\gamma(t)) - \inf \beta)^2}{36}((|x|-|\gamma(t)|)^2+2t) - 2(\lambda - \nu_1/a^2) } \, |\dot\gamma(t)| dt \right) ,   \] 
as claimed.
\end{proof}

\begin{proof}[Proof of Prop \eqref{thm:carmona-bdd}]
If $\beta(x)$ is as in the hypotheses, then $\beta+U=W\in L^p(\bbR^2)$ for $p>2$. Thus by the $L^p$ Hodge-Kodaira decomposition \cite[Thm 1.2]{troyanov2009hodge} we have the existence of $A=(A_1,A_2)\in L^p(\bbR^2;\bbR^2)$, satisfying that $\Div A=0$ weakly, and solving $\text{curl}(A)=\beta+U$. From this we conclude that $|A|^2\in L^{p/2}(\bbR^2)$; since $p/2>1$ we can conclude that the corresponding potential $A\in \vec\cK_{\loc}(\bbR^2)$ as needed to apply the Feynman-Kac-It\'o formula, Thm \ref{thm:broderix}.

Fix $a>0$ and let $y$ be a point achieving $W(y)=\essinf_{y_0\in \bbR^2} W(y_0)+a$. As we saw in the proof of Thm \eqref{thm:main}, we have that for any $C^1$ path $\gamma(t):[0,T]\to \Omega$ from $x$ to $y$, that
\[ \left| \bbE\left[ e^{\lambda T + i\int_0^T A(\omega(t))\circ d\omega(t)} \right] \right| = e^{\lambda T} \exp\left\{ -\tfrac12 \int_0^T \bar\beta_x(\gamma(t)) dt \right\}   \] 
which when combined with the Feynman-Kac-It\'o formula \ref{thm:broderix} implies
\[ \frac{|f(x)|}{\|f\|_\infty}  \leq e^{\lambda T} \exp\left\{ -\tfrac12 \int_0^T \left(\int_0^1 s\, \bbE[\beta(\Gamma(s,t))]\, ds\right)^2 \bbE[|\gamma(t)-\omega(t)|^2]\, dt \right\} .  \] 
Now using the splitting $\beta = W - U$ where $W\in L_{\loc}^p(\bbR^2)$ for $p>2$, and $U\geq 0$ and $U\in L^\infty(\bbR^2)$ we note
\[ \left(\int_0^1 s\, \bbE[\beta(\Gamma(s,t))]\, ds\right)^2 \geq  \left(\int_0^1 s\, \bbE[W_\infty - U(\Gamma(s,t))] ds \right)^2 \geq \frac{(W_\infty-\|U\|_{L^\infty})^2}{2} \] 
Now if $A=\{\omega: \sup_{0\leq t\leq T} |\gamma(t)-\omega(t)|<a\}$ then 
\begin{align*}
& \left| \bbE\left[ e^{\lambda T + i\int_0^T A(\omega(t))\cdot d\omega(t)} (1_A+1_{A^c}) \right] \right|  \\
& \hspace{1.5cm} \leq e^{\lambda T} \left(e^{-\tfrac12 \int_0^T \bar\beta(\gamma(t))\, 1_A(\omega(t))\, dt} \bbE[1_A] + e^{-\tfrac{T}{2}W_a^\gamma(x)}\right) \\
& \hspace{1.5cm} \leq e^{\lambda T} \left( C_a e^{-\tfrac{a^2T}{2} (W_\infty-\|U\|_{L^\infty})^2} e^{-\tfrac{\nu_1T}{a^2}-\tfrac12 \int_0^T |\dot\gamma(t)|^2 dt} + e^{-\tfrac{T}{2}W_a^\gamma(x)}  \right) \\
& \hspace{1.5cm} \leq e^{\lambda T} \left( C_a e^{-\tfrac{a^2T}{2}(W_\infty-\|U\|_{L^\infty})^2} e^{-\tfrac{\nu_1T}{a^2}-\tfrac{\ell(\gamma)^2}{2T}} + e^{-\tfrac{T}{2}W_a^\gamma(x)}  \right)
\end{align*} 
where in the penultimate inequality we have used as before that 
\[ \bbE[1_A] = C\exp\left(-\frac{\nu_1T}{a^2} - \tfrac12\int_0^T |\dot\gamma(t)|^2dt\right), \] 
and in the final inequality used the fact that $\ell(\gamma)^2\leq 2T\,  E(\gamma)$.
\end{proof}

\section{Acknowledgements}
We are very grateful to Stefan Steinerberger, Andrea Ottolini, and Robert Hingtgen for helpful discussion and comments throughout this project.

\bibliographystyle{abbrv} 
\bibliography{reference} 

\end{document}